\documentclass[11pt,a4paper]{article}
\usepackage[utf8]{inputenc}

\usepackage{hyperref}
\usepackage{mymacros3}

\usepackage{amsmath}
\usepackage{bm}
\usepackage{amssymb}
\usepackage{color}
\usepackage{mathtools}
\usepackage{yhmath}
\usepackage[all]{xy}

\usepackage{cleveref}
\usepackage[margin=3.5cm]{geometry}

\usepackage[
backend=biber,
style=ext-alphabetic,
maxnames = 10,
maxalphanames = 10,
maxcitenames  = 10,
mincitenames  = 4,
minalphanames = 4,
minnames      = 1,
sorting=nyt,
giveninits,
articlein=false,
url=true,  
doi=false,
eprint=true,
isbn=false
]{biblatex}
\title{Tamarkin's separation theorem for non-compact objects \\ in cotangent bundles}
\author{Yuichi Ike \and Tatsuki Kuwagaki}
\date{\today}

\begin{document}

\maketitle

\begin{abstract}
In this short note, we prove a Tamarkin-type separation theorem for possibly non-compact subsets in cotangent bundles.
\end{abstract}

\section{Introduction}

Let $M$ be a manifold and denote by $\pi \colon T^*M \to M$ its cotangent bundle. Let $\bK$ be a unital commutative ring.
Let $\bG$ be a subgroup of the discrete additive group $\bR$ of the real numbers and denote by $\Lambda_0^\bG$ the Novikov ring associated with $\bG$. We consider the additive action of $\bG$ on $M\times \bR_{u<c}\times \bR_t$ on the rightmost factor and the trivial action on the other factor. We denote the equivariant derived category of $\bK$-module sheaves on $M\times \bR_{u<c}\times \bR_t$ by $\Sh^\bG(M\times \bR_{u<c}\times \bR_t)$. We set
\begin{multline}
    \Sh_{\tau>0}^\bG(M\times\bR_{u<c}\times \bR_t) \\
    \coloneqq \Sh^\bG(M\times\bR_{u<c}\times \bR_t)/\lc \cE\in\Sh^\bG(M\times\bR_{u<c}\times \bR_t)\relmid \SS(\cE)\subset \lc \tau\leq 0\rc \rc,     
\end{multline}
where $\tau$ is the cotangent coordinate of $T^*\bR_t$ given by $dt$ and $\SS$ is the microsupport.

For $c \in (0,\infty]$, we denote by $\mu^\bG(T^*M;u<c)$ the microlocal category introduced in \cite{IK}, which is defined as the subcategory of $\Sh_{\tau>0}^\bG(M\times\bR_{u<c}\times \bR_t)$ spanned by the weak doubling movies.
Here, $\cE\in \Sh_{\tau>0}^\bG(M\times\bR_{u<c}\times \bR_t)$ is said to be a weak doubling movie if there exists a $\bG$-invariant subset $A\subset T^*M\times T^*_{\tau>0}\bR_t$ such that $\SS(\cE)\subset AA$, where
\begin{equation}\label{eq:doubling}
    \begin{split}
        AA &\coloneqq AA_h\cup AA_t\subset T^*M \times T^*\bR_{u<c} \times T^*_{\tau>0}\bR_t,\\
        AA_h &\coloneqq \lc (p,u,0,t,\tau) \relmid (p, t, \tau) \in A, u \ge 0 \rc \subset T^*M \times T^*\bR_{u<c} \times T^*_{\tau >0}\bR_t,\\
        AA_t &\coloneqq \lc (p, u, \upsilon, t, \tau) \relmid (p, t, \tau) \in A, u \ge 0, \upsilon=-\tau\rc \subset T^*M \times T^*\bR_{u<c} \times T^*_{\tau>0}\bR_t.     
    \end{split}
    \end{equation}
The microlocal category $\mu^\bG(T^*M;u<c)$ is a $\Lambda_0^\bG$-linear category.
For an object $\cE \in \mu^\bG(T^*M;u<c)$, its non-conic microsupport is denoted by $\musupp(\cE) \subset T^*M$.
For a closed subset $A'$ of $T^*M$, we define 
\begin{equation}
    \mu_{A'}^\bG(T^*M;u<c) \coloneqq \lc \cE\in \mu^\bG(T^*M;u<c) \relmid \musupp(\cE)\subset A'\rc. 
\end{equation}
A closed subset $A'$ of $T^*M$ is said to be \emph{end-conic} if there exists a disk bundle $D^*M$ in $T^*M$ cut out by some fiber metric such that $A' \setminus D^*M$ is conic, i.e., stable under the scaling action of $\bR_{\geq 1}$.

In this paper, we prove the following separation theorem, which is a slight generalization of the separation theorem for the original Tamarkin category~\cite{Tamarkin, GS14}.

\begin{theorem}\label{theorem:separation}
    Let $A'$ and $B'$ be end-conic closed subsets $T^*M$.
    Suppose that $\pi(A') \cap \pi(B')$ is compact and $A' \cap B' = \varnothing$.
    Then for any $\cE \in \mu_{A'}^\bG(T^*M;u<c)$ and any $\cF \in \mu_{B'}^\bG(T^*M;u<c)$, one has $\Hom_{\mu^\bG(T^*M;u<c)}(\cE,\cF)=0$.    
\end{theorem}

We prepare some lemmas in \cref{subsec:cutoff} and give a proof in \cref{subsec:proof_separtion}.

\subsection*{Notation}

Throughout the paper, we fix a unital integral commutative ring $\bK$.
Let $X$ be a manifold and let $\pi \colon T^*X \to X$ denote its cotangent bundle.
We write $T^*_XX$ for the zero-section of $T^*X$.
We also let $\partial T^*X$ denote the contact boundary defined as $(T^*X \setminus T^*_XX)/\bR_{>0}$. 

We denote by $\bK_X$ the constant sheaf on $X$ with stalk $\bK$.
A \emph{category} (resp.\ \emph{triangulated category}) means a \emph{dg-category} (resp.\ \emph{pre-triangulated dg-category}) unless specified.
For example, $\Sh(X)$ the derived category of the abelian category of $\bK_X$-modules, is a triangulated category in our sense.
For an object $\cE \in \Sh(X)$, we write $\MS(\cE)$ for the microsupport of $\cE$ (see \cite{KS90} for the definition), which is a closed conic subset of $T^*X$.

\section{Cut-off result}\label{subsec:cutoff}

In this subsection, we prove the following: 
For $\cE \in \mu^\bO(T^*M;u<c)$, $\musupp(\cE) \subset T^*M$ controls the whole microsupport $\MS(\cE) \subset T^*(M \times \bR_{u<c} \times \bR_t)$, where $\bO$ denotes the trivial subgroup of $\bR$ and $\cE$ is regarded as an object of $\Sh(M \times \bR_{u<c} \times \bR_t)$ through the projector $(-)\star \bK_{\geq 0}$.

First, we give a cut-off lemma in the special case.
Let $V$ be a finite-dimensional real vector space and $\gamma$ be a closed convex proper cone in $V$ with $0 \in \gamma$ and $\gamma \neq \{0\}$.
We also set $U_\gamma \coloneqq T^*M \times V \times \Int \gamma^\circ$ and $Z_\gamma \coloneqq T^*M \times T^*V \setminus U_\gamma$ as in \cite{GS14}, where $\gamma^\circ$ denotes the polar cone of $\gamma$: $\{ \theta \in V^* \mid \langle \theta,v \rangle \ge 0 \ \text{for any $v \in \gamma$} \}$.

We recall known cut-off results from \cite{GS14}, which we will use in the following proofs.
\begin{lemma}[{\cite[Prop.~4.17(ii)]{GS14}}]\label{lemma:cutoffGS}
    For $\cF \in \Sh(\bK_V)$, we have
    \begin{equation}
        \MS(\cF \star \bK_{\gamma})\subset (\MS(\cF)\cap U_\gamma)\cup V\times \partial \gamma^\circ.
    \end{equation}
    In particular,
    \begin{equation}
        \MS(\cF \star \bK_{\gamma})\cap U_\gamma\subset (\MS(\cF) \cap U_\gamma).
    \end{equation}
\end{lemma}

We slightly generalize this and \cite[Lem.~4.25]{GS14} as follows.

\begin{lemma}\label{lemma:cutoff_pts}
    Let $\cF \in {}^\perp \Sh_{Z_\gamma}(\bK_{V})$ and assume that there exists a closed cone $C^* \subset V^*$ such that 
    \begin{enumerate}
    \renewcommand{\labelenumi}{$\mathrm{(\arabic{enumi})}$}
        \item $C^* \subset \gamma^\circ$ and
        \item $\MS(\cF) \cap U_ \gamma \subset V \times \Int C^*$.
    \end{enumerate}
    Then one has $\MS(\cF) \subset (\MS(\cF) \cap U_ \gamma) \cup (V \times (C^* \cap \partial \gamma^\circ))$.
\end{lemma}

\begin{proof}
    We mimic the proof of \cite[Lem.~4.25]{GS14}.
    We first claim the following.

    \begin{claim}
        For any proper closed convex cone $\beta$ of $V$, we have 
    \begin{equation}
        \MS(\cF) \cap U_\beta 
        \subset \lb (\MS(\cF) \cap U_\gamma) \cup V \times (\Conv(C^* \cap \beta^\circ) \cap \partial \gamma^\circ) \rb \cap U_\beta.
    \end{equation}
    \end{claim}

    \begin{proof}
    Let us consider the sheaf $\cG = \cF \star \bK_\beta$. 
    By \cref{lemma:cutoffGS}, we have 
    \begin{equation}
        \MS(\cG) \cap U_\beta = \MS(\cF) \cap U_\beta,
    \end{equation}
    which implies 
    \begin{equation}
        \MS(\cG) \cap U_\gamma \cap U_\beta 
        \subset 
        V \times \Int C^* \cap U_\beta
    \end{equation}
    by the condition~(2).
    
    Set $\lambda \coloneqq (C^* \cap \beta^\circ)^\circ \subset V$. 
    We shall show
        \begin{equation}\label{eq:cutoff2}
        \MS(\cG \star \bK_\lambda) \cap U_\beta
        \subset \lb (\MS(\cG) \cap U_\gamma) \cup V \times (\Conv(C^* \cap \beta^\circ) \cap \partial \gamma^\circ) \rb \cap U_\beta
    \end{equation}
    and
    \begin{equation}\label{eq:cutoff1}
        \cG \cong \cG \star \bK_{\lambda}.
    \end{equation}    

    Since $\cG \in {}^\perp \Sh_{Z_\gamma}(\bK_V)$, by \cref{lemma:cutoffGS}, we have $\MS(\cG)=\MS(\cG \star \bK_\gamma) \subset V \times \gamma^\circ$ and
    \begin{equation}
    \begin{aligned}
        \MS(\cG) \cap U_\beta
        & = \MS(\cG \star \bK_\gamma) \cap U_\beta\\
        & \subset 
        ((\MS(\cG) \cap U_\gamma) \cup (V \times \partial \gamma^\circ)) \cap U_\beta\\
        & \subset 
        V \times ((\Int \lambda^\circ \cup \partial \gamma^\circ) \cap \Int \beta^\circ ).
    \end{aligned}
    \end{equation}
    In particular, we have $\MS(\cG) \subset V \times (\Int \lambda^\circ \cup \partial \gamma^\circ \cup \partial \beta^\circ )$ by \cref{lemma:cutoffGS}.
    Since $\cG \star \bK_\lambda \cong (\cG \star \bK_\gamma) \star \bK_\lambda$, by the microsupport estimate for $\star$, we have 
    \begin{equation}
    \begin{aligned}
        \MS(\cG \star \bK_\lambda) \cap U_\beta 
        & \subset 
        V \times (\lambda^\circ \cap (\Int \lambda^\circ \cup \partial \gamma^\circ \cup \partial \beta^\circ ) \cap \Int \beta^\circ ) \\
        & \subset 
        V \times (\Int \lambda^\circ \cup (\lambda^\circ \cap \partial \gamma^\circ) \cap \Int \beta^\circ) \\
        & \subset 
        (U_\gamma \cup (V \times (\Conv(C^* \cap \beta^\circ) \cap \partial \gamma^\circ))) \cap U_\beta,
    \end{aligned}   
    \end{equation}
    which proves \eqref{eq:cutoff2}.    
    Moreover, since 
    \begin{equation}
    \begin{aligned}
        \MS(\cG \star \bK_{\lambda \setminus \gamma}) \cap U_\beta
        & \subset V \times ((\gamma^\circ \setminus \Int \lambda^\circ) \cap (\Int \lambda^\circ \cup \partial \gamma^\circ \cup \partial \beta^\circ) \cap \Int \beta^\circ) \\
        & \subset V\times (\partial \gamma^\circ\cap \Int \beta^\circ)\subset Z_\gamma\cap U_\beta.
    \end{aligned}
    \end{equation}
    
    From the exact triangle 
    \begin{equation}
        \cG \star \bK_{\lambda \setminus \gamma} \to \cG \star \bK_\lambda \to \cG \star \bK_\gamma \xrightarrow{+1},
    \end{equation}
    we find that $\cG \star \bK_\lambda \to \cG \star \bK_\gamma$ is an isomorphism in $\Sh(\bK_V;U_\gamma \cap U_\beta)=\Sh(\bK_V;U_{(\gamma^\circ\cap \beta^\circ)^\circ})$ by the above microsupport estimate. 
    Since $\cG = \cG \star \bK_\gamma \star \bK_\beta$, we have $\MS(\cG)\subset V\times (\gamma^\circ\cap \beta^\circ)$. 
    Hence we have $\cG \star \bK_{(\gamma^\circ\cap \beta^\circ)^\circ}\cong \cG$. 
    Through the equivalence $ (-)  \star \bK_{(\gamma^\circ\cap \beta^\circ)^\circ} \colon \Sh(\bK_V;U_{(\gamma^\circ\cap \beta^\circ)^\circ})\rightarrow {}^\perp\Sh_{Z_{(\gamma^\circ\cap \beta^\circ)^\circ}}(\bK_V)$, we have an isomorphism $\cG \star \bK_\lambda\rightarrow \cG$ in $\Sh(\bK_V)$. 
    Hence, we obtain \eqref{eq:cutoff1}, which proves the claim.
    \end{proof}

    Let us return to the proof of \cref{lemma:cutoff_pts}. 
    If we set $\beta=\{0\}$ in the claim, we get 
    \begin{equation}
        \MS(\cF) 
        \subset  (\MS(\cF) \cap U_\gamma) \cup V \times (\Conv(C^*) \cap \partial \gamma^\circ).
    \end{equation}
    Suppose $(x,\xi) \in \MS(\cF) \cap V \times ((\Conv(C^*) \setminus C^*) \cap \partial \gamma^\circ)$. 
    Then, we can find a proper closed convex cone $\beta$ such that $\xi \in \Int \beta^\circ$ and $C^* \cap \beta^\circ=\{0\}$.
    By applying the claim, we get 
    \begin{equation}
        (x,\xi) \in \MS(\cF) \cap V \times \Int \beta^\circ 
        \subset \lb (\MS(\cF) \cap U_\gamma) \cup V \times 0 \rb \cap V \times \Int \beta^\circ,
    \end{equation}
    which is a contradiction. 
    This proves \cref{lemma:cutoff_pts}.
\end{proof}

\begin{remark}
    In the above lemma, if $C^*$ is a strict $\gamma$-cone in the sense of \cite{GS14}, then $C^* \cap \partial \gamma^\circ=\{0\}$. 
    Hence, the lemma generalizes \cite[Lem.~4.25]{GS14}.
\end{remark}

\begin{proposition}\label{lemma:cutoff_manifold}
    Let $M$ be an open subset of $E=\bR^d$ and $\cF \in {}^\perp \Sh_{Z_\gamma}(\bK_{M \times V})$.
    Assume that there exists a closed cone $C^* \subset E^* \times V^*$ such that
    \begin{enumerate}
    \renewcommand{\labelenumi}{$\mathrm{(\arabic{enumi})}$}
        \item $C^* \subset E^* \times \gamma^\circ$ and
        \item $\MS(\cF) \cap U_ \gamma \subset (M \times V) \times C^*$.
    \end{enumerate}
    Then one has $\MS(\cF) \subset (\MS(\cF) \cap U_ \gamma) \cup (M \times V) \times (C^* \cap (E^* \times \partial \gamma^\circ))$.
\end{proposition}

\begin{proof}
    The proof is similar to \cite[Thm.~4.27]{GS14}.
    We replace \cite[Lem.~4.25]{GS14} with \cref{lemma:cutoff_pts}.
    The statement is local on $M$.
    Let $x_0 \in M$ and $K$ be a compact neighborhood of $x_0$. 
    We choose an open neighborhood $W$ of $K$ and a diffeomorphism $W \cong \bR^d=E$ satisfying $(-1,1)^d \subset K$.
    We take a diffeomorphism $\varphi \colon (-1,1) \simto \bR$ such that $d\varphi(t) \ge 1$ for any $t \in (-1,1)$.
    Define $\Phi \colon U \coloneqq (-1,1)^d \times V \simto E \times V$ by 
    \begin{equation}
        \Phi(x'_1,\dots,x'_d, x'') \coloneqq (\varphi(x'_1),\dots,\varphi(x'_d),x''). 
    \end{equation}
    Then, we obtain $\Phi_\pi \Phi_d^{-1}(U \times C^*) \subset E \times C^*$ (see \cite[Lem.~4.26]{GS14}).
    We can apply \cref{lemma:cutoff_pts} to the sheaf $\Phi_*(\cF|_U)$ on the vector space $E \times V$ with the cone $\{0\} \times \gamma$ and any neighborhood closed cone $\widetilde C^*$ of $C^*$ in $E^*\times\gamma^\circ$. We then obtain the estimate $\MS(\cF) \subset (\MS(\cF) \cap U_ \gamma) \cup (M \times V) \times (\widetilde C^* \cap (E^* \times \partial \gamma^\circ))$. By taking the intersections over such $\widetilde C^*$'s, we obtain the desired estimate.
\end{proof}

For an end-conic closed subset $A'$ of $T^*M$, we set $\partial A' \coloneqq A' \cap \partial T^*M \subset \partial T^*M$. 
Set 
\begin{equation}
    \rho \colon T^*M \times T^*_{\tau >0}\bR_t \to T^*M; (p,t,\tau) \mapsto p/\tau.
\end{equation}

\begin{proposition}\label{proposition:ms_representative}
    Let $A'$ be an end-conic closed subset of $T^*M$ and set $A=\rho^{-1}(A')$. 
    Then for any $\cE \in \mu^{\bO}_{A'}(T^*M;u<c)$, one has 
    \begin{equation}
        \MS(\cE|_{u>0}) \subset AA \cup \{ (x,a\xi,u,0,t,0) \mid (x,\xi) \in \partial A',a>0 \} \cup T^*_{M \times \bR_{u<c} \times \bR_t}(M \times \bR_{u<c} \times \bR_t),
    \end{equation}
    where $AA \subset T^*M \times T^*\bR_{u<c} \times T^*_{\tau>0}\bR_t$ denotes the doubling of $A$ (see \cref{eq:doubling}).
\end{proposition}

\begin{proof}
    Fix $(x,u)\in M\times \bR_{u<c}$ and take a coordinate compact neighborhood $K$.
    Then the cotangent bundle is trivialized over $K$, namely, $T^*(M\times \bR_{u<c})|_K \cong K \times V^*$ with $V=\bR^{n+1}$. 
    Set $A=\rho^{-1}(A')$.
    For each $(x',u') \in K$, we set
    \begin{equation}
        C_{(x',u')}^* \coloneqq \overline{T^*_{(x',u')}(M\times \bR_{u<c}) \times \bR_\tau\cap AA} \subset V^* \times \bR_\tau.
    \end{equation}
    We also define
    \begin{equation}
    \begin{aligned}
        C_K^* & \coloneqq \bigcup_{(x',u')\in K}C_{(x',u')}^* \subset V^* \times \bR_\tau, \\
        C^{\infty,*}_{(x',u')} & \coloneqq \bR_{> 0}\cdot (\partial A'\cap \partial T^*_{x'} M) \times \{ \upsilon=0\} \times \{\tau=0\} \cup \{(0,0)\} \subset V^* \times \bR_\tau. 
    \end{aligned}
    \end{equation}
    Since $A'$ is end-conic, we find that $C_K^* \cap \{\tau=0 \} \subset \bigcup_{(x',u') \in K}C^{\infty,*}_{(x',u')}$.

    We take an open subset $U \subset K$ containing $(x,u)$ and set $\cE_{U}\coloneqq \cE|_{U \times \bR_t}$. 
    Then we have $\SS(\cE_{U}) \cap \{\tau>0\} \subset (U \times \bR_t) \times C_K^*$. 
    Hence, by applying \cref{lemma:cutoff_manifold} to $\cE_{U}$ with $V=\bR_t$ and $\gamma=\bR_{\ge 0}$, we obtain
    \begin{equation}
        \SS(\cE_{U}) \subset (\SS(\cE_{U}) \cap U_\gamma) \cup ((U\times \bR_t) \times (C_K^* \cap \{ \tau =0\})).
    \end{equation}
    Here, $\SS(\cE_{U}) \cap U_\gamma \subset AA$ and $C_K^* \cap \{\tau=0\} \subset \bigcup_{(x',u') \in K}C^{\infty,*}_{(x',u')}$.
    In particular, 
    \begin{equation}
        \SS(\cE)\times_{\lc(x,u)\rc \times \bR_t}T^*(M\times \bR^2) \subset AA \cup  (U\times \bR_t) \times \bigcup_{(x',u') \in K}C^{\infty,*}_{(x',u')}
    \end{equation}
    for any coordinate compact neighborhood $K$ of $(x,u)$. 
    By taking the intersection over the neighborhoods of $x$, we get the desired estimate over $x$. 
    This completes the proof.
\end{proof}

\section{Proof of separation theorem}\label{subsec:proof_separtion}

We first prove the separation theorem in the non-equivariant case.

\begin{proposition}\label{proposition:sep_nonequiv}
    Let $A'$ and $B'$ be end-conic closed subsets $T^*M$.
    Suppose that $\pi(A') \cap \pi(B')$ is compact and $A' \cap B' = \varnothing$.
    Then for any $\cE \in \mu^{\bO}_{A'}(T^*M;u<c)$ and any $\cF \in \mu^{\bO}_{B'}(T^*M;u<c)$, one has $\Hom_{\mu^{\bO}(T^*M;u<c)}(\cE,\cF)=0$.
\end{proposition}

\begin{proof}  
    We will prove that $\Hom_{\Sh(M \times \bR_{u<c} \times \bR_t)}(\cE_{M \times (-\infty,a) \times \bR_t},\cF)=0$ for $a>0$. 
    If the claim holds,
    since $\colim_{a \nearrow c} \cE_{M\times (-\infty, a)\times \bR_t}=\cE_{M\times \bR_{u<c}\times \bR_t}$, we have
    \begin{equation}
        \Hom_{\Sh(M \times \bR_{u<c} \times \bR_t)}(\cE_{M \times \bR_{u<c} \times \bR_t},\cF)\cong \lim_{a \nearrow c}\Hom_{\Sh(M \times \bR_{u<c} \times \bR_t)}(\cE_{M \times (-\infty,a) \times \bR_t},\cF)\cong 0.
    \end{equation}
    Hence, we will check the claim below.

    We argue similarly to \cite[Lem.~5.15]{AISQ} to estimate $\MS(\cE_{M \times (-\infty,a) \times \bR_t})$.
    Since $\MS(\cE|_{u>0}) \cap \{ \tau=0\} \subset \{ \upsilon=0\}$ and $N^*(M \times (0,a) \times \bR_t) \subset \{ \xi=0, \tau =0 \}$, we find that $\MS(\cE) \cap N^*(M \times (0,a) \times \bR_t) \subset T^*_{M \times \bR_{u<c} \times \bR_t}(M \times \bR_{u<c} \times \bR_t)$. 
    Noticing that $\cE|_{u \leq 0} \cong 0$ and applying \cite[Prop.~5.4.8]{KS90}, we get 
    \begin{equation}
        \MS(\cE_{M \times (-\infty,a) \times \bR_t}) \subset N^*(M \times (0,a) \times \bR_t)^a + \MS(\cE).
    \end{equation}
    By \cref{proposition:ms_representative}, we have 
    \begin{equation}
        \MS(\cE_{M \times (-\infty,a) \times \bR_t}) \cap \{ u=a\}
        \subset  
        \begin{aligned}
            & \{ (x,\xi,a,\upsilon,t,\tau) \mid (x,\xi,t-a,\tau) \in A, \upsilon \ge -\tau \} \\
            & \cup \{ (x,\xi,a,\upsilon,t,0) \mid (x,\xi) \in \partial A', \upsilon \ge 0 \}, 
        \end{aligned}
    \end{equation}
    where $A=\rho^{-1}(A')$.
    Since $\cE_{M \times (-\infty,a) \times \bR_t} \cong \bK_{M \times \bR_{u<c} \times [0,\infty)} \star \cE_{M \times (-\infty,a) \times \bR_t}$, we have isomorphisms 
    \begin{align}
        & \Hom_{\Sh(M \times \bR_{u<c} \times \bR_t)}(\cE_{M \times (-\infty,a) \times \bR_t}, \cF) \\
        \cong {} & \Hom_{\Sh(M \times \bR_{u<c} \times \bR_t)}(\bK_{M \times \bR_{u<c} \times [0,\infty)} \star \cE_{M \times (-\infty,a) \times \bR_t}, \cF) \\
        \cong {} & \Hom_{\Sh(M \times \bR_{u<c} \times \bR_t)}(\bK_{M \times \bR_{u<c} \times [0,\infty)}, \cHom^\star(\cE_{M \times (-\infty,a) \times \bR_t}, \cF)).
    \end{align}
    Here, $\cHom^\star$ is the right adjoint of $\star$.
    By \cite[Lem.~4.10]{GS14}, $\cHom^\star$ can be written as 
    \begin{equation}
        \cHom^\star(\cG,\cH) 
        \cong 
        m_* \cHom(p_2^{-1}i^{-1}\cG, p_1^!\cH),
    \end{equation}
    where 
    \begin{equation}
    \begin{aligned}
        & p_i \colon M \times \bR_{u<c} \times \bR^2 \to M \times \bR_{u<c} \times \bR_t; \quad (x,u,t_1,t_2) \mapsto (x,u,t_i), \\
        & m \colon M \times \bR_{u<c} \times \bR^2 \to M \times \bR_{u<c} \times \bR_t; \quad (x,u,t_1,t_2) \mapsto (x,u,t_1+t_2), \\
        & i \colon M \times \bR_{u<c} \times \bR_t \to M \times \bR_{u<c} \times \bR_t; \quad (x,u,t) \mapsto (x,u,-t).
    \end{aligned}
    \end{equation}
    By adjunction, we find that $q_* \cHom^\star(\cE_{M \times (-\infty,a) \times \bR_t}, \cF) \in \Sh_{\{\tau \leq 0\}}(\bR_t)^\perp$, where $q \colon M \times \bR_{u<c} \times \bR_t\rightarrow \bR_t$ is the projection. 
    Since $\MS(i^{-1}\cE_{M \times (-\infty,a) \times \bR_t}) \cap \MS(\cF) \subset T^*_{M \times \bR_{u<c} \times \bR_t}(M \times \bR_{u<c} \times \bR_t)$, 
    by the microsupport estimate, we have 
    \begin{equation}
    \begin{aligned}
        & \SS(\cHom^\star(\cE_{M \times (-\infty,a) \times \bR_t}, \cF)) \\
        \subset {} & 
        \left\{ 
            (x,\xi,u,\upsilon,t,\tau) \; \middle| \;
            \begin{aligned}
                & \text{there exist } (\xi_1,\upsilon_1), (\xi_2,\upsilon_2) \in T^*_{(x,u)}(M \times \bR_{u<c}) \\
                & \text{and } t_1, t_2 \in \bR_t \text{ with } t=t_1+t_2 \text{ such that} \\
                & (x,\xi_1,u,\upsilon_1,t_1,\tau) \in \SS(\cE_{M \times (-\infty,a) \times \bR_t}), \\ 
                & (x,\xi_2,u,\upsilon_1,t_2,\tau) \in \SS(\cF), \\
                & \xi = -\xi_1+\xi_2, \text{ and } \upsilon=-\upsilon_1+\upsilon_2
            \end{aligned}
        \right\}.
    \end{aligned}
    \end{equation}
    Again by \cref{proposition:ms_representative}, we obtain 
    \begin{equation}\label{eq:separation_estimate}
    \begin{aligned}
        \SS(\cHom^\star(\cE_{M \times (-\infty,a) \times \bR_t}, \cF)) \cap {} & (T^*_{M \times \bR_{u<c}}(M \times \bR_{u<c}) \times T^*\bR_t) \\
        & \subset T^*_{M \times \bR_{u<c} \times \bR_t}(M \times \bR_{u<c} \times \bR_t).
    \end{aligned}    
    \end{equation}
    Since the map $q$ is proper on $\Supp(\cHom^\star(\cE_{M \times (-\infty,a) \times \bR_t}, \cF))$, we get 
    \begin{equation}
        \MS(q_* \cHom^\star(\cE_{M \times (-\infty,a) \times \bR_t}, \cF)) \subset 0_{\bR_t}.
    \end{equation}
    By combining this estimate with the fact that it is in $\Sh_{\{\tau \leq 0\}}(\bR_t)^\perp$, we find that $q_* \cHom^\star(\cE_{M \times (-\infty,a) \times \bR_t}, \cF) \cong 0$. 
    This implies 
    \begin{equation}
    \begin{aligned}
        & \Hom_{\Sh(M \times \bR_{u<c} \times \bR_t)}(\cE_{M \times (-\infty,a) \times \bR_t}, \cF) \\
        \cong {} &  
        \Hom_{\Sh(\bR_t)}(\bK_{[0,\infty)}, q_* \cHom^\star(\cE_{M \times (-\infty,a) \times \bR_t}, \cF)) \cong 0
    \end{aligned}
    \end{equation}
    as desired.
\end{proof}

Finally, we prove the separation theorem for our equivariant microlocal category. For the operations $\star_\bG$ and $\cHom^{\star_\bG}$, we refer to \cite{IK}.

\begin{proof}[Proof of \cref{theorem:separation}]
    We have 
    \begin{equation}
    \begin{aligned}
        & \Hom_{\mu^\bG(T^*M;u<c)}(\cE, \cF) \\ 
        \cong {} & \Hom_{\mu^\bG(T^*M;u<c)}(\bigoplus_{d \in \bG}\bK_{M \times \bR_{u<c} \times [d,\infty)} \star_\bG \cE, \cF) \\
        \cong {} & \Hom_{\mu^\bG(T^*M;u<c)}(\bigoplus_{d \in \bG}\bK_{M \times \bR_{u<c} \times [d,\infty)} ,\cHom^{\star_\bG}(\cE, \cF)) \\
        \cong {} & \Hom_{\Sh(M\times \bR_{u<c}\times \bR_t)}(\bK_{M \times \bR_{u<c} \times [d,\infty)} ,\frakf(\cHom^{\star_\bG}(\cE, \cF)) \\
        \cong {} & \Hom_{\Sh^{(\bG)}(M\times \bR_{u<c}\times \bR_t)}(\bK_{M \times \bR_{u<c} \times [d,\infty)} , {p_1}_* \cHom_{M \times \bR^2}(p_2^{-1}\cF, m^!\cG)).
    \end{aligned}
    \end{equation}
Here $\Sh^{(\bG)}(M\times \bR_{u<c}\times \bR_t)$ is the derived category of equivariant sheaves with respect to the trivial $\bG$-action. The last space is $\bG$-invariant of the $\bG$-representation $\Hom_{\Sh(M\times \bR_{u<c}\times \bR_t)}(\bK_{M \times \bR_{u<c} \times [c,\infty)} , {p_1}_* \cHom_{M \times \bR^2}(p_2^{-1}\cF, m^!\cG))$. Hence the vanishing follows from that of the non-equivariant version (\cref{proposition:sep_nonequiv}).
\end{proof}

\section*{Acknowledgments}

The first author was supported by JSPS KAKENHI Grant Numbers JP21K13801.
The second was supported by JSPS KAKENHI Grant Numbers JP22K13912 and JP20H01794.
The first author thanks Tomohiro Asano for the helpful discussions. 

\printbibliography

\noindent Yuichi Ike:
Department of Mathematics, University of Tokyo, 3-8-1 Komaba Meguro-ku, Tokyo 153-8914, Japan

\noindent
\textit{E-mail address}: \texttt{ike[at]ms.u-tokyo.ac.jp},
\medskip

\noindent Tatsuki Kuwagaki: 
Department of Mathematics, Kyoto University, Kitashirakawa Oiwake-cho, Sakyo-ku, Kyoto 606-8502, Japan.

\noindent 
\textit{E-mail address}: \texttt{tatsuki.kuwagaki.a.gmail.com}

\end{document}